\documentclass[12pt,a4paper]{article}
\usepackage[utf8]{inputenc}
\usepackage[english]{babel}
\usepackage{xcolor}

\usepackage{amssymb,mathtools,amsmath,amsthm,stmaryrd}
\usepackage[all,cmtip,2cell]{xy} 

\usepackage[draft=false, hidelinks, linkcolor = blue]{hyperref}
\usepackage{cleveref}
\usepackage{tensor}
\usepackage{bbm}
\usepackage{url}
\usepackage{bookmark}
\usepackage{footmisc}

\usepackage{graphicx}
\usepackage{lmodern}
\usepackage{enumitem}
\usepackage{array}

\usepackage[left=2cm,right=2cm,top=2cm,bottom=2cm]{geometry}
\def\defthm#1#2#3#4{
  \newtheorem{#1}[theorem]{#3}
  \newtheorem*{#1*}{#3}
  \newtheorem{#2}[theorem]{#4}
  \newtheorem*{#2*}{#4}
  \crefname{#1}{#3}{#4}
  \crefname{#2}{#4}{#4}  
}

\newtheoremstyle{mythm}% 
{10pt}% Space above 
{}% Space below 
{\itshape}% Body font 
{}% Indent amount 
{\bf}%  Theorem head font 
{.}% Punctuation after theorem head 
{.5em}% Space after theorem head 
{}% 

\newtheoremstyle{mydef}% 
{10pt}% Space above 
{3pt}% Space below 
{}% Body font 
{}% Indent amount 
{\bf}%  Theorem head font 
{.}% Punctuation after theorem head 
{.5em}% Space after theorem head 
{}% 

\newtheoremstyle{myrmk}% 
{10pt}% Space above 
{3pt}% Space below 
{}% Body font 
{}% Indent amount 
{\bf}%  Theorem head font 
{.}% Punctuation after theorem head 
{.5em}% Space after theorem head 
{}% 

\theoremstyle{mythm}
\newtheorem{theorem}{Theorem}[section]
\newtheorem*{theorem*}{Theorem}

\defthm{corollary}{corollaries}{Corollary}{Corollaries}
\defthm{lemma}{lemmata}{Lemma}{Lemmata}
\defthm{proposition}{propositions}{Proposition}{Propositions}
\defthm{axiom}{axioms}{Axiom}{Axioms}
\defthm{propdef}{props/defs}{Proposition/Definition}{Prositions/Definitions}
\defthm{defcor}{defscors}{Corollary/Definition}{Corollaries/Definitions}
\defthm{conjecture}{conjectures}{Conjecture}{Conjectures}

\theoremstyle{mydef}
\defthm{definition}{definitions}{Definition}{Definitions}

\theoremstyle{myrmk}
\defthm{acknowledgment}{acknowledgments}{Acknowledgment}{Acknowledgments}
\defthm{remark}{remarks}{Remark}{Remarks}
\defthm{motivation}{motivations}{Motivation}{Motivations}
\defthm{example}{examples}{Example}{Examples}
\defthm{question}{questions}{Question}{Questions}
\defthm{observation}{observations}{Observation}{Observations}
\defthm{claim}{claims}{Claim}{Claims}
\defthm{notation}{notations}{Notation}{Notations}
\defthm{terminology}{terminologies}{Terminology}{Terminologies}
\defthm{construction}{constructions}{Construction}{Constructions}
\defthm{conclusion}{conclusions}{Conlusion}{Conlusions}

\newcommand{\Address}{{% additional braces for segregating \footnotesize
  \bigskip
  \footnotesize

\textsc{Max Planck Institute for Mathematics, Vivatsgasse 7, 53111 Bonn, Germany}\par\nopagebreak
  \textit{E-mail address}: \texttt{stenzel@mpim-bonn.mpg.de}
}}

\title{Lurie's Unstraightening as a weak biequivalence of $\infty$-cosmoses}
\author{Raffael Stenzel}
\renewcommand\footnotemark{}

\begin{document}
\maketitle

\begin{abstract}
We give a direct proof of the fact that Lurie's Unstraightening functor induces an equivalence between the strict
$(\infty,2)$-category of indexed quasi-categories and the strict $(\infty,2)$-category of fibered quasi-categories over any 
given quasi-categorical base. We conclude that Unstraightening preserves simplicial cotensors up to a (strictly) 
natural homotopy equivalence, and thus gives rise to an accordingly weakened notion of cosmological biequivalence between the 
two underlying $\infty$-cosmoses.
\end{abstract}
%\tableofcontents

\section{Introduction}

%For title: What I am proving is that Unstraightening is essentially surjective and a local equivalence. Need to be able to construct an 
%inverse by formal reasons, which then follows to be equivalent to Straightening.
%\begin{terminology*}
%For the sake of readability and conformity with the common conventions, the term ``$\infty$-category`` shall mean
%``$(\infty,1)$-category'' throughout this paper.
%\end{terminology*}

This note briefly addresses a rather subtle aspect of Lurie's Unstraightening construction. To motivate its interest, 
we recall that the Grothendieck construction in ordinary category theory (over some 1-categorical base $\mathcal{C}$) is an 
equivalence 
\[\int\colon\mathbf{Fun}(\mathcal{C}^{op},\mathrm{Cat})\rightarrow\mathbf{GFib}(\mathcal{C})\]
between the 2-category of (pseudo-functorial) $\mathcal{C}$-indexed categories and the 2-category of Grothendieck fibrations over
$\mathcal{C}$. See e.g.\ \cite[Theorem 8.3.1]{borceuxhandbook2}. As such it yields an equivalence between the theory of indexed 
categories (over some base) and the theory of fibered categories (over that base).
%Indeed, note that a category theory is the practice in a 2-category, not a 1-category.

Accordingly, one expects the $\infty$-categorical analogon to the Grothendieck construction -- presented by Lurie's Unstraightening 
functor -- to constitute an equivalence between indexed and fibered $\infty$-category theory, in the sense that it induces an 
equivalence of respective $(\infty,2)$-categories. 
By construction, Straightening/Unstraightening (over a quasi-categorical base $\mathcal{C}$) is defined as an adjunction 
\begin{align}\label{equstunadj}
(\mathrm{St},\mathrm{Un})\colon\mathbf{S}^+_{/\mathcal{C}^{\sharp}}\rightarrow\mathbf{Fun}(\mathfrak{C}(\mathcal{C})^{op},\mathbf{S}^+)
\end{align}
between the category of marked simplicial sets sliced over the (maximally marked) quasi-category $\mathcal{C}^{\sharp}$ and the 
category of contravariant simplicially enriched functors from its freely generated simplicial category $\mathfrak{C}(\mathcal{C})$ into 
the simplicially enriched category $\mathbf{S}^+$ of marked simplicial sets.
This adjunction is shown in \cite[Theorem 3.2.0.1]{luriehtt} to be a Quillen equivalence of model categories if one equips the left 
hand side with the cartesian model structure ``$\mathrm{Cart}$'' over $\mathcal{C}^{\sharp}$ (\cite[Section 3.1.3]{luriehtt}) and the 
right hand side with the projective model structure ``$\mathrm{Proj}$'' with respect to the cartesian model structure on $\mathbf{S}^+$ (i.e.\ 
on the slice $\mathbf{S}^+_{/(\Delta^0)^{\sharp}}$). As such, it induces an equivalence
\[(\mathrm{St},\mathrm{Un})\colon\mathrm{Cart}(\mathcal{C})\rightarrow\mathrm{Fun}(\mathcal{C}^{op},\mathrm{Cat}_{\infty})\]
between the underlying $\infty$-category of cartesian fibrations over $\mathcal{C}$ and the underlying $\infty$-category of
$\mathcal{C}$-indexed $\infty$-categories. 

Additionally, both model categories $(\mathbf{S}^+_{/\mathcal{C}^{\sharp}},\mathrm{Cart})$ and
$(\mathbf{Fun}(\mathfrak{C}(\mathcal{C})^{op},(\mathbf{S}^+,\mathrm{Cart})),\mathrm{Proj})$ exhibit an enrichment over the model 
category $(\mathrm{S},\mathrm{QCat})$ for quasi-categories (\cite[Remark 3.1.4.5]{luriehtt} and \cite[Remark A.3.3.4]{luriehtt}). They 
both can hence be assigned an underlying strict $(\infty,2)$-category $\mathbf{Cart}(\mathcal{C})$ and
$\mathbf{Fun}(\mathcal{C}^{op},\mathbf{Cat}_{\infty})$ each, given by the full simplicial category spanned by the 
respective bifibrant objects. Here, a strict $(\infty,2)$-category is defined to be a simplicially enriched category whose hom-objects 
are quasi-categories (\cite[Section 5.5.8]{kerodon}).
Now, if the adjunction (\ref{equstunadj}) was to be simplicially enriched itself, it would induce an according equivalence 
of underlying strict $(\infty,2)$-categories by a fairly straight-forward formal argumentation. However, the Straightening functor is 
not a simplicially enriched functor, which essentially follows from the fact that the unmarked Straightening functor over the point is 
not (isomorphic to) the identity \cite[Remark 2.2.2.6]{luriehtt}. This failure is exemplified by the fact that Straightening generally 
does not preserve simplicial tensors up to natural isomorphism. In particular, the adjunction (\ref{equstunadj}) is not simplicially 
enriched either. Lurie shows in \cite[Corollary 3.2.1.15]{luriehtt} that Straightening however does preserve simplicial tensors up to a 
natural homotopy equivalence. This implies that at least its right adjoint can be simplicially enriched after all, as has been observed 
right away in \cite[Section 3.2.4]{luriehtt}. There is however no mention in \cite{luriehtt} of the right adjoints behaviour with 
regards to simplicial cotensors, presumably because the book avoids $(\infty,2)$-categorical considerations in general. And the 
dual of \cite[Corollary 3.2.1.15]{luriehtt} as stated in Theorem~\ref{thmunstraight+} is not a formal triviality, and neither 
is the fact that Unstraightening indeed induces a DK-equivalence of underlying strict $(\infty,2)$-categories.

Thus, in Section~\ref{secproof} we give the evident arguments which show that Unstraightening yields a functor of strict
$(\infty,2)$-categories, that it induces equivalences of derived hom-quasi-categories, and that the dual of
\cite[Corollary 3.2.1.15]{luriehtt} holds. Regarding the latter, we recall that Unstraightening cannot preserve simplicial cotensors 
up to natural isomorphism either as this would again imply that the entire Straightening/Unstraightening adjunction is simplicially enriched. The objective of this paper is hence to give a short proof of the following statement.

\begin{theorem}\label{thmunstraight+}
Let $\mathcal{C}$ be a quasi-category. The Unstraightening functor over $\mathcal{C}$ exhibits a simplicial enrichment
\[\mathbf{Un}_{\mathcal{C}}\colon\mathbf{Fun}(\mathfrak{C}(\mathcal{C})^{op},(\mathbf{S}^+,\mathrm{Cart}))_{\mathrm{proj}}\rightarrow(\mathbf{S}^+_{/\mathcal{C}},\mathrm{Cart})\] 
which induces a DK-equivalence of underlying strict $(\infty,2)$-categories. It furthermore comes equipped with a binatural 
transformation
\[\mathrm{Un}_{\mathcal{C}}(F^I)\rightarrow\mathrm{Un}_{\mathcal{C}}(F)^I\]
between the respective simplicial cotensors for $F\colon\mathfrak{C}(\mathcal{C})^{op}\rightarrow\mathbf{S}^+$ and $I\in\mathbf{S}$, 
which is a cartesian equivalence whenever $F$ is projectively fibrant. 
\end{theorem}

In Section~\ref{seccosmo} we discuss Theorem~\ref{thmunstraight+} from an $\infty$-cosmological point of view.

\paragraph{Corresponding results in the literature.}
%The theorem concerns the specific presentation of Unstraightening as a functor between strict $(\infty,2)$-categories as defined in 
%\cite{luriehtt}. Other presentations of the construction in the context of other models of higher category theory (other than the model 
%of quasi-categories, that is) have been considered in the literature, and in some cases has been stated to in fact give rise to an 
%equivalence of according $(\infty,2)$-categories indeed. It for instance follows most directly from 
%Prop. 3.15 and Corollary 3.90 of Abellan Stern - 2-cartesian fibrations II: Unstraightening induces an equivalence of
%$(\infty,2)$-categories (modelled as fibrant scaled simplicial sets).	this however is a corollary of a very long and involved series of 
%papers and does not yield the strictness stuff.

The fact that Unstraightening induces an equivalence of underlying $(\infty,2)$-categories (modelled as fibrant scaled simplicial sets) 
follows as a special case of the fact that the more general (and more involved) Unstraightening construction of $\infty$-categories 
indexed over an $(\infty,2)$-category as introduced in \cite{lgood} always induces an equivalence of underlying
$(\infty,2)$-categories. This is shown in \cite[Lemma 1.4.3]{ghl2fib}. Albeit obscured by a considerable amount of additional theory 
and a few non-trivial details omitted, %(such as according versions of Lemma~\ref{lemmaunstraight1} and Corollary~\ref{corunstraight1}), 
the more general argument given there is essentially the same as the argument give here. Thus, the new insight this note provides in 
this respect is merely marginal.

Almost simultaneously, the fact that Straightening preserves simplicial tensors up to homotopy equivalence has been used in
\cite[Proposition 6.9]{ghnlaxlimits} to show that it induces an equivalence of the two underlying $\infty$-categories also when 
equipped with their canonical enhanced mapping $\infty$-categories (\cite[Section 6]{ghnlaxlimits}). This statement however does not 
quite show that Straightening is an equivalence of $(\infty,2)$-categories, because it is not shown to be a functor of $(\infty,2)$-
categories in the first place. To elaborate, the $\infty$-category with enhanced mapping $\infty$-categories which underlies a given
$(\infty,2)$-category forgets the horizontal composition operation and hence all according horizontal coherences for all higher cells. 
A crucial consequence of this is that, although Unstraightening is the inverse of Straightening (as functors of underlying
$\infty$-categories) there is no formal reason why this would readily imply that Unstraightening induces equivalences of enhanced 
mapping $\infty$-categories as well.
%If we were to know that 
%Straightening was $(\infty,2)$-functorial, we would obtain an according $(\infty,2)$-categorical inverse by formal reasoning, which 
%consequently would have to coincide with Unstraightening up to equivalence on underlying $\infty$-categories.

\begin{acknowledgments*}
This note was written while being a guest at the Max Planck Institute for Mathematics in Bonn, Germany, whose hospitality is greatly 
appreciated. 
\end{acknowledgments*}

%The statement of the theorem is a slight misrepresentation of the way we will prove it. We will first prove the existence of a binatural 
%transformation $\mathrm{Un}(F^I)\rightarrow\mathrm{Un}(F)^I$ in the $\infty$-category
%$\mathrm{Cat}_{\infty}\times\mathrm{Fun}(\mathcal{C}^{op},\mathrm{Cat}_{\infty})$ which we argue to be a pointwise equivalence 
%(Proposition~\ref{}). We do so by constructing said natural equivalence over the base $\mathrm{Cat}_{\infty}$ first, and then use that 
%all constructions are pullback-stable. This proves the general case as $\pi$ is universal.
%We then deduce that Unstraightening induces equivalences of hom-quasi-categories (Theorem~\ref{}). Given that it 
%is essentially surjective, it follows that it is a biequivalence of $(\infty,2)$-categories. Lastly, we use Theorem~\ref{} to show that 
%there in fact is a binatural equivalence $\mathrm{Un}(F^I)\rightarrow\mathrm{Un}(F)^I$ in the simplicial category
%$\mathbf{S}\times\mathbf{Fun}(\mathcal{C}^{op},\mathbf{S}^+)$ (Corollary~\ref{}).

\section{Brief reminder of the underlying constructions}

We recall Lurie's Unstraightening construction (\cite[Section 3.2]{luriehtt}). Therefore, first recall that the category
$\mathbf{S}^+$ of marked simplicial sets is cartesian closed, and hence is particular simplicially enriched via the canonical forgetful 
functor $U\colon\mathbf{S}^+\rightarrow\mathbf{S}$. This yields the ``flat'' simplicial enrichment of 
\cite[Section 3.1.3]{luriehtt} in the following sense: The functor $U$ has both a left adjoint
$(\cdot)^{\flat}\colon\mathbf{S}\rightarrow\mathbf{S}^+$ and a right adjoint $(\cdot)^{\sharp}\colon\mathbf{S}\rightarrow\mathbf{S}^+$. 
The simplicial 
enrichment $\mathbf{S}^+(X,Y):=U(Y^X)$ exhibits simplicial tensors in $\mathbf{S}^+$ via the formula $I\otimes X:=I^{\flat}\times X$. 
The simplicial category $\mathbf{S}^+$ furthermore may be equipped with the cartesian model structure ``$\mathrm{Cart}$''.
Its cofibrations are exactly the cofibrations of underlying simplicial sets; in particular, all marked simplicial sets are cofibrant.
Its fibrant objects are exactly the marked simplicial sets $I^{\natural}$ where $I$ is a quasi-category and $I^{\natural}$ denotes $I$ 
marked by its set of equivalences. Furthermore, the cartesian model structure on $\mathbf{S}^+$ is
enriched over the Joyal model structure $(\mathbf{S},\mathrm{QCat})$ for quasi-categories as well, and the forgetful functor
$(\mathbf{S}^+,\mathrm{Cart})\rightarrow(\mathbf{S},\mathrm{QCat})$ is a simplicially enriched right Quillen 
equivalence \cite[Proposition 3.1.5.3]{luriehtt}. It hence induces a simplicial functor on the simplicially enriched categories
$(\mathbf{S}^+,\mathrm{Cart})^f\rightarrow\mathrm{QCat}$ of (bi)fibrant objects which in fact is an isomorphism. In the 
following, we denote the simplicially enriched category of fibrant objects associated to the $(\mathbf{S},\mathrm{QCat})$-enriched model 
category $(\mathbf{S}^+,\mathrm{Cart})$ by $\mathbf{Cart}$. More generally, for a simplicial set $S$, the simplicially enriched category 
of fibrant objects associated to the simplicially enriched category $\mathbf{S}^+_{/S}$ equipped with its according cartesian model 
structure ``$\mathrm{Cart}$'' is denoted by $\mathbf{Cart}(S)$.

Now, Unstraightening over a quasi-category $\mathcal{C}$ is a functor
\begin{align}\label{equdefunst}
\mathrm{Un}_{\mathcal{C}}\colon\mathbf{Fun}(\mathfrak{C}(\mathcal{C})^{op},\mathbf{S}^+)\rightarrow\mathbf{S}^+_{/\mathcal{C}^{\sharp}}.
\end{align}
It is the right adjoint part of a Quillen equivalence between the projective model structure (with respect to
$(\mathbf{S}^+,\mathrm{Cart})$) on the left hand side and the cartesian model structure over $\mathcal{C}^{\sharp}$ on the right hand 
side. Both model structures are $(\mathbf{S},\mathrm{QCat})$-enriched and hence give rise to an underlying strict
$(\infty,2)$-category each which is defined as the respective quasi-categorically enriched full subcategory of bifibrant objects.

Generally, a strict $(\infty,2)$-category \cite[Section 5.5.8]{kerodon} is a simplicially enriched category whose hom-objects are
quasi-categories. A functor of strict $(\infty,2)$-categories is just a simplicially enriched functor between such. A functor of strict
$(\infty,2)$-categories is a DK-equivalence of strict $(\infty,2)$-categories if it is essentially surjective on associated
homotopy-categories and induces equivalences between hom-quasi-categories.

\section{Proof of the Theorem}\label{secproof}

In this section we prove Theorem~\ref{thmunstraight+}. Therefore, we recall the existence of a natural weak cartesian 
equivalence $q\colon\mathrm{St}_{\ast}\Rightarrow 1$ from the Straightening functor
$\mathrm{St}_{\ast}\colon\mathbf{S}^+\rightarrow\mathbf{S}^+$ over a point to the identity on $\mathbf{S}^+$
\cite[Proposition 3.2.1.14]{luriehtt}. Given that the tuple $(\mathrm{St}_{\ast},\mathrm{Un}_{\ast})$ forms an adjoint pair, we obtain 
a mate of the form
\begin{align}\label{equdefcompcotensors}
\bar{q}\colon 1\xRightarrow{\eta}\mathrm{Un}_{\ast}\mathrm{St}\xRightarrow{\mathrm{Un}_{\ast}(q)}\mathrm{Un}_{\ast}.
\end{align}

As noted in \cite[Proposition 2.15]{hhr_str}, we obtain the following formal consequence of \cite[Proposition 3.2.1.14]{luriehtt}.

\begin{lemma}\label{lemmaunstraight1}
The map $\bar{q}_{X}\colon X\rightarrow\mathrm{Un}_{\ast}(X)$ is a cartesian equivalence whenever the 
marked simplicial set $X\in\mathbf{S}^+$ is fibrant in the marked model structure.
\end{lemma}
\begin{proof}
Let $\rho\colon 1\Rightarrow\mathbb{R}$ be a fibrant replacement functor in $(\mathbf{S}^+,\mathrm{Cart})$. For every $X\in\mathbf{S}^+$ 
we obtain a diagram in $\mathbf{S}^+$ as follows.
\[\xymatrix{
X\ar[d]_{\eta_X} \ar@/^/[drr]^{\sim} & & \\
\mathrm{Un}_{\ast}\mathrm{St}_{\ast}(X)\ar[rr]^{\mathrm{Un}_{\ast}\rho_{\mathrm{St}_{\ast}(X)}}\ar[d]_{\mathrm{Un}_{\ast}(q_X)} & & \mathrm{Un}_{\ast}\mathbb{R}\mathrm{St}_{\ast}(X)\ar[d]^{\mathrm{Un}_{\ast}\mathbb{R}(q_X)}_[@!-90]{\sim} \\
\mathrm{Un}_{\ast}(X)\ar[rr]_{\mathrm{Un}_{\ast}\rho_X} & & \mathrm{Un}_{\ast}\mathbb{R}(X)
}\]
The curved arrow on top is the derived unit at $X$. Since $(\mathrm{St}_{\ast},\mathrm{Un}_{\ast})$ is a Quillen equivalence, it is a 
weak cartesian equivalence for all (cofibrant) $X\in\mathbf{S}^+$. The vertical arrow $\mathrm{Un}_{\ast}\mathbb{R}(q_X)$ is a cartesian 
equivalence because $q_X$ is a weak cartesian equivalence. The bottom horizontal arrow $\mathrm{Un}_{\ast}(\rho_X)$ is a cartesian 
equivalence whenever $X$ is fibrant (as $\mathrm{Un}_{\ast}$ is a right Quillen functor). By 2-out-of-3 it follows that the left 
vertical composition $\bar{q}_X\colon X\rightarrow\mathrm{Un}_{\ast}(X)$ is a cartesian equivalence whenever $X$ is fibrant.
\end{proof}

\begin{corollary}\label{corunstraight1}
The map $\bar{q}_{I^{\flat}}\colon I^{\flat}\rightarrow\mathrm{Un}_{\ast}(I^{\flat})$ is a cartesian equivalence for all
quasi-categories $I$.
\end{corollary}
\begin{proof}
We first note that the canonical inclusion $\iota\colon I^{\flat}\rightarrow I^{\natural}$ in $\mathbf{S}^+$ computes a fibrant 
replacement of $I^{\flat}$: It is a cofibration of marked simplicial sets since the identity on the simplicial set $I$ is a cofibration. 
Furthermore, the fibrant objects in $(\mathbf{S}^+,\mathrm{Cart})$ are exactly the marked simplicial sets of the form $J^{\natural}$ for 
quasi-categories $J$. To prove acyclicity of the inclusion, it suffices to show that it has the left lifting property against all 
fibrations in $(\mathbf{S}^+,\mathrm{Cart})$ between fibrant objects (\cite[Lemma 7.14]{jtqcatvsss}). Thus, we are to show that every 
square of the form
\[\xymatrix{
I^{\flat}\ar[r]^{f}\ar@{^(->}[d] & J^{\natural}\ar@{->>}[d] \\
I^{\natural}\ar[r]^{g} & K^{\natural}
}\]
in $\mathbf{S}^+$ for quasi-categories $I$, $J$ and $K$ admits a lift. This however follows from the fact that the underlying maps of 
simplicial sets are all functors of quasi-categories and hence preserve equivalences.

As $\bar{q}$ is a natural transformation, we obtain a square of marked simplicial sets as follows.
\begin{align}\label{diagcorunstraight1}
\begin{gathered}
\xymatrix{
I^{\flat}\ar[r]^(.4){\bar{q}_{I^{\flat}}}\ar@{^(->}[d]_{\iota} & \mathrm{Un}_{\ast}(I^{\flat})\ar@{^(->}[d]^{\mathrm{Un}_{\ast}(\iota)} \\
I^{\sharp}\ar[r]_(.4){\bar{q}_{I^{\natural}}} & \mathrm{Un}_{\ast}(I^{\natural})
}
\end{gathered}
\end{align}
Its bottom map is a weak cartesian equivalence by Lemma~\ref{lemmaunstraight1} as $I$ is a quasi-category by assumption.
The left vertical inclusion is acyclic as we just observed.
The right vertical map $\mathrm{Un}_{\ast}(\iota)$ is the canonical inclusion of
$\mathrm{Un}_{\ast}(I^{\flat})=\mathrm{Un}_{\ast}(I)^{\flat}$ in $\mathrm{Un}_{\ast}(I^{\natural})=\mathrm{Un}_{\ast}(I)^{\natural}$. 
Indeed, for any marking $E$ on $I$, the underlying simplicial set of the Unstraightening $\mathrm{Un}_{\ast}(I,E)$ is the (unmarked) 
Unstraightening $\mathrm{Un}_{\ast}(I)$ by construction (as explicitly stated in \cite[Section 2.4]{hhr_str}).
Although the general simplicial structures diverge, the vertices and edges of $\mathrm{Un}_{\ast}(I)$ are exactly the vertices and edges 
of $I$ (with the same boundaries and degeneracies); 
under this identification, the edges marked in $\mathrm{Un}_{\ast}(I,E)$ for any set of edges $E\in I_1$ are exactly the edges in $E$. 
It follows that the inclusion $\mathrm{Un}_{\ast}(\iota)\colon\mathrm{Un}_{\ast}(I)^{\flat}\hookrightarrow\mathrm{Un}_{\ast}(I)^{\natural}$ is acyclic as well. Thus, 
by 2-out-of-3, the top map in the square (\ref{diagcorunstraight1}) is a weak cartesian equivalence, too.
\end{proof}

Furthermore, we record the following (fairly specific) generalization of the well-known fact that a simplicial enrichment of a Quillen 
pair allows to compute its derivations simply by restriction to the simplicially enriched categories of bifibrant objects 
(\cite[Proposition 5.2.4.6]{luriehtt}).
% whenever all objects in the codomain are cofibrant 

\begin{lemma}\label{lemmainftyadjsemi}
Let $\mathbf{M}$ and $\mathbf{N}$ be simplicial model categories and (for brevity) suppose all objects in $\mathbf{M}$ are cofibrant. 
Suppose
$\mathbf{G}\colon\mathbf{N}\rightarrow\mathbf{M}$ is a simplicially enriched functor such that its underlying functor
$G\colon\mathbf{N}_0\rightarrow\mathbf{M}_0$ is a right Quillen functor. Then the homotopy-coherent nerve
$N_{\Delta}(G)\colon N_{\Delta}(\mathbf{N})\rightarrow N_{\Delta}(\mathbf{M})$ computes the natural functor of homotopy
$\infty$-categories associated to $G$. In particular, for all cofibrant objects $A\in\mathbf{N}$ and all fibrant objects
$B\in\mathbf{N}$ the simplicial action
\begin{align}\label{equlemmainftyadjsemi}
\mathbf{G}(A,B)\colon\mathbf{N}(A,B)\rightarrow\mathbf{M}(G(A),G(B))
\end{align}
is homotopy equivalent to its action
\[\mathbb{R}(G)(A,B)\colon\mathbf{N}_0(A,B)^h\rightarrow\mathbf{M}_0(G(A),G(B))^h\]
on derived mapping spaces. This implies that the action (\ref{equlemmainftyadjsemi}) is an equivalence of hom-spaces whenever $G$ is a 
right Quillen equivalence.
\end{lemma}
\begin{proof}
For any simplicially enriched model category $\mathbf{M}$, the canonical inclusion
$\mathbf{M}_0^{\mathrm{cf}}\hookrightarrow\mathbf{M}^{\mathrm{cf}}$ associated to its full simplicial subcategory of bifibrant objects 
induces an inclusion of the category 
$N_{\Delta}(\mathbf{M}_0^{\mathrm{cf}})=N(\mathbf{M}_0^{\mathrm{cf}})$ into the quasi-category
$N_{\Delta}(\mathbf{M}^{\mathrm{cf}})$. This inclusion of quasi-categories presents the $\infty$-categorical 
localization of $N(\mathbf{M}_0^{\mathrm{cf}})$ at the class of weak equivalences by a folklore combination of results in the literature 
(summarized in \cite{cisinskimoinftyloc}). We are thus to show under the given assumptions that the square
\[\xymatrix{
N(\mathbf{N}_0^{\mathrm{cf}})\ar[r]^{N(G)}\ar@{^(->}[d] & N(\mathbf{M}_0^{\mathrm{cf}})\ar@{^(->}[d] \\
N(\mathbf{N}^{\mathrm{cf}})\ar[r]_{N_{\Delta}(\mathbf{G})} & N(\mathbf{M}^{\mathrm{cf}})
}\]
commutes (up to equivalence). This however follows directly from the trivial fact that the square
\[\xymatrix{
\mathbf{N}_0^{\mathrm{cf}}\ar[r]^{G}\ar@{^(->}[d] & \mathbf{M}_0^{\mathrm{cf}}\ar@{^(->}[d] \\
\mathbf{N}^{\mathrm{cf}}\ar[r]_{\mathbf{G}} & \mathbf{M}^{\mathrm{cf}}
}\]
commutes by assumption. For every pair of objects $A,B\in\mathbf{M}^{cf}$, the counit of the $(\mathfrak{C},N_{\Delta})$-adjunction induces a weak equivalence
\[\xymatrix{
\mathbf{N}^{\mathrm{cf}}(A,B)\ar[rr]^{\mathbf{G}(A,B)} & & \mathbf{M}^{\mathrm{cf}}(A,B) \\
\mathfrak{C}N_{\Delta}(\mathbf{N}^{\mathrm{cf}})(A,B)\ar[rr]_{\mathfrak{C}N_{\Delta}(\mathbf{G})(A,B)}\ar[u]^{\epsilon(A,B)} & & \mathfrak{C}N_{\Delta}(\mathbf{M}^{\mathrm{cf}})(A,B)\ar[u]_{\epsilon(A,B)}
}\]
of maps of spaces. The bottom map is equivalent to the functor $N_{\Delta}\mathbf{G}(A,B)$ via \cite[Proposition 2.2.2.7]{luriehtt} and 
\cite[Proposition 2.2.4.1]{luriehtt}. 
\end{proof}

The following proposition describes a simplicial action of the Unstraightening functor. Therefore, we exploit
%, first, that the underlying quasi-category functor $U\colon\mathrm{Cat}_{\infty}(\mathcal{S})\rightarrow\mathrm{Cat}_{\infty}$ is the right derivation of a right Quillen equivalence
%$U\colon(s\mathbf{S},\mathrm{Cat}_{\infty})\rightarrow(\mathbf{S},\mathrm{QCat})$ (\cite{jtqcatvsss} under the name $\iota^{\ast}_1$). 
%Similarly, the converse equivalence
%$((\cdot)^{\bullet})^{\simeq}\colon\mathrm{Cat}_{\infty}\rightarrow\mathrm{Cat}_{\infty}(\mathcal{S})$ is the right 
%derivation of a right Quillen equivalence $((\cdot)^{\bullet})^{\simeq}\colon\rightarrow(s\mathbf{S},\mathrm{Cat}_{\infty})$
%(\cite{jtqcatvsss} under the name $t^!$). The composition $U\circ((\cdot)^{\bullet})^{\simeq}\colon\mathbf{S}\rightarrow\mathbf{S}$ is 
%isomorphic to the identity as shown in the proof of \cite[Theorem 4.12]{jtqcatvsss}). In particular, for any
%simplicial category $\mathbf{C}$ and any pair of objects $C,D\in\mathbf{C}$ the simplicial hom-set $\mathbf{C}(C,D)$ is naturally 
%isomorphic to the simplicial set $U((\mathbf{C}(C,D)^{\Delta^{\bullet}})^{\simeq})$. If $\mathbf{C}$ is tensored over $\mathbf{S}$, the 
%latter in turn is naturally isomorphic to the simplicial set $U(\mathbf{C}(C\otimes\Delta^{\bullet},D)^{\simeq})$.
%
that both simplicial categories $\mathbf{Fun}(\mathfrak{C}(\mathcal{C})^{op},\mathbf{S}^+)$ and $\mathbf{S}^+_{/\mathcal{C}}$ are
tensored over $\mathbf{S}$ in a very particular way: On the one hand, given a functor
$F\colon\mathfrak{C}(\mathcal{C})^{op}\rightarrow\mathbf{S}^+$ and a 
simplicial set $I$, the functor $I\otimes F\colon\mathfrak{C}(\mathcal{C})^{op}\rightarrow\mathbf{S}^+$ is defined as the product 
$c_{I^{\flat}}\times F$ for $c_{I^{\flat}}$ the constant functor with value $I^{\flat}\in\mathbf{S}^+$. On the other hand, given a map 
$f\colon X\rightarrow\mathcal{C}^{\sharp}$ of marked simplicial sets and a simplicial set $I$, the object
$I\otimes f\in\mathbf{S}^+_{/\mathcal{C}}$ is defined as the product $\pi_{I^{\flat}}\times_{\mathcal{C}} f$ for
$\pi_{I^{\flat}}\colon \mathcal{C}^{\sharp}\times I^{\flat}\rightarrow\mathcal{C}^{\sharp}$ the obvious projection. Dually this implies 
that the corresponding simplicial cotensors are in fact exponentials and hence $\infty$-categorical constructs as well.

\begin{proposition}\label{propunstraight1}
The Unstraightening functor exhibits a simplicial action
\[\mathbf{Un}_{\mathcal{C}}(F,G)\colon\mathbf{Fun}(\mathfrak{C}(\mathcal{C})^{op},\mathbf{S}^+)(F,G)\rightarrow\mathbf{S}^+_{/\mathcal{C}}(\mathrm{Un}(F),\mathrm{Un}(G))\]
of hom-objects which is an equivalence of (derived) hom-quasi-categories whenever $F$ is projectively cofibrant and $G$ is projectively 
fibrant (with respect to the cartesian model structure on $\mathbf{S}^+$). In particular, Unstraightening induces a simplicially 
enriched functor
\[\mathbf{Un}_{\mathcal{C}}\colon\mathbf{Fun}(\mathfrak{C}(\mathcal{C})^{op},\mathbf{S}^+)\rightarrow\mathbf{S}^+_{/\mathcal{C}}\] 
which induces a DK-equivalence of underlying strict $(\infty,2)$-categories.
\end{proposition}
\begin{proof}
In the following for the sake of readability, we denote the simplicial hom-set of two functors
$F,G\in\mathbf{Fun}(\mathfrak{C}(\mathcal{C})^{op},\mathbf{S}^+)$ by $\mathbf{Nat}(F,G)$. We have a chain of binatural transformations 
of simplicial sets as follows.
\begin{align}
\notag \mathbf{Nat}(F,G) & \cong (\mathbf{Nat}(F,G))^{\Delta^{\bullet}})_0\\
\notag & \cong\mathbf{Nat}(\Delta^{\bullet}\otimes F,G)_0 \\
\notag & \cong \mathbf{Nat}(c_{(\Delta^{\bullet})^{\flat}}\times F,G)_0 \\
\label{equpropunstraight11} & \xrightarrow{\mathrm{Un}_{\mathcal{C}}} \mathbf{S}^+_{/\mathcal{C}}(\mathrm{Un}_{\mathcal{C}}(c_{(\Delta^{\bullet})^{\flat}}\times F),\mathrm{Un}_{\mathcal{C}}(G))_0 \\
\label{equpropunstraight12} & \cong \mathbf{S}^+_{/\mathcal{C}}(\mathrm{Un}_{\mathcal{C}}(c_{(\Delta^{\bullet})^{\flat}})\times\mathrm{Un}_{\mathcal{C}}(F),\mathrm{Un}_{\mathcal{C}}(G))_0 \\
\label{equpropunstraight13} & \cong\mathbf{S}^+_{/\mathcal{C}}(\pi_{\mathrm{Un}_{\ast}((\Delta^{\bullet})^{\flat})}\times \mathrm{Un}_{\mathcal{C}}(F),\mathrm{Un}_{\mathcal{C}}(G))_0 \\
\label{equpropunstraight14} & \xrightarrow{(\pi_{\bar{q}_{(\Delta^{\bullet})^{\flat}}}\times 1)^{\ast}} \mathbf{S}^+_{/\mathcal{C}}(\pi_{(\Delta^{\bullet})^{\flat}}\times\mathrm{Un}_{\mathcal{C}}(F),\mathrm{Un}_{\mathcal{C}}(G))_0 \\
\notag & \cong \mathbf{S}^+_{/\mathcal{C}}(\Delta^{\bullet}\otimes\mathrm{Un}_{\mathcal{C}}(F),\mathrm{Un}_{\mathcal{C}}(G))_0 \\
\notag & \cong (\mathbf{S}^+_{/\mathcal{C}}(\mathrm{Un}_{\mathcal{C}}(F),\mathrm{Un}_{\mathcal{C}}(G))^{\Delta^{\bullet}})_0 \\
\notag & \cong\mathbf{S}^+_{/\mathcal{C}}(\mathrm{Un}_{\mathcal{C}}(F),\mathrm{Un}_{\mathcal{C}}(G))
\end{align}

Here, the arrow in (\ref{equpropunstraight11}) is the natural action of Unstraightening as a 1-functor. The natural map in Line 
(\ref{equpropunstraight12}) is an isomorphism by the fact that Unstraightening is a right adjoint and hence preserves products.
The natural isomorphism in Line (\ref{equpropunstraight13}) is an instance of pullback-stability of the Unstraightening construction
\cite[Observation 2.13]{hhr_str}. Lasty, the action in Line (\ref{equpropunstraight14}) is given by the pullback of the natural 
transformation $\bar{q}|_{\Delta}\colon (\Delta^{\bullet})^{\flat}\rightarrow\mathrm{Un}_{\ast}((\Delta^{\bullet})^{\flat})$ from 
(\ref{equdefcompcotensors}) along $\mathcal{C}^{\sharp}\rightarrow 1$. We obtain a composite simplicial action
\begin{align}\label{equpropunstraightcomp}
\mathbf{Un}_{\mathcal{C}}(F,G)\colon\mathbf{Nat}(F,G)\rightarrow\mathbf{S}^+_{/\mathcal{C}}(\mathrm{Un}_{\mathcal{C}}(F),\mathrm{Un}_{\mathcal{C}}(G))
\end{align}
which returns $\mathrm{Un}_{\mathcal{C}}(F,G)$ on sets of vertices.

Whenever $G$ is projectively fibrant, then $\mathrm{Un}_{\mathcal{C}}(G)$ is fibrant in $(\mathbf{S}^+_{/\mathcal{C}},\mathrm{Cart})$. 
The object $\mathrm{Un}_{\mathcal{C}}(F)$ is always cofibrant in $(\mathbf{S}^+_{/\mathcal{C}},\mathrm{Cart})$. It follows that whenever 
$F$ is projectively cofibrant and $G$ is projectively fibrant, both $\mathbf{Nat}(F,G)$ and
$\mathbf{S}^+_{/\mathcal{C}}(\mathrm{Un}_{\mathcal{C}}(F),\mathrm{Un}_{\mathcal{C}}(G))$ are (the respectively derived)
hom-quasi-categories. We are left to show that the action (\ref{equpropunstraightcomp}) is a categorical equivalence in this case. 
Therefore, it suffices to show that the induced functor
\[(\mathbf{Un}_{\mathcal{C}}(F,G)^{\Delta^{\bullet}})^{\simeq}\colon(\mathbf{Nat}(F,G)^{\Delta^{\bullet}})^{\simeq}\rightarrow(\mathbf{S}^+_{/\mathcal{C}}(\mathrm{Un}_{\mathcal{C}}(F),\mathrm{Un}_{\mathcal{C}}(G))^{\Delta^{\bullet}})^{\simeq}\]
of complete Segal spaces is a pointwise equivalence of spaces. The latter is by construction naturally isomorphic to the composition
\[\mathbf{Nat}(\Delta^{\bullet}\otimes F,G)^{\simeq}\xrightarrow{\mathbf{Un}_{\mathcal{C}}(\Delta^{\bullet}\otimes F,G)^{\simeq}}\mathbf{S}^+_{/\mathcal{C}}(\mathrm{Un}_{\mathcal{C}}(\Delta^{\bullet}\otimes F),\mathrm{Un}_{\mathcal{C}}(G))^{\simeq}\xrightarrow{(\pi_{\bar{q}_{(\Delta^{\bullet})^{\flat}}}\times 1)^{\ast}}\mathbf{S}^+_{/\mathcal{C}}(\Delta^{\bullet}\otimes\mathrm{Un}_{\mathcal{C}}(F),\mathrm{Un}_{\mathcal{C}}(G))^{\simeq}.\]
Now, first, the left component of this composition is a pointwise equivalence of hom-spaces. Indeed,
$\mathbf{Fun}(\mathfrak{C}(\mathcal{C})^{op},\mathbf{S}^+)$ is a simplicial model category when considered to be equipped with the
hom-objects $\mathbf{Nat}(F,G)^{\simeq}$. Thus we may use that $\mathrm{Un}_{\mathcal{C}}$ is a right Quillen equivalence as shown in 
\cite[Theorem 3.2.0.1]{luriehtt} and apply Lemma~\ref{lemmainftyadjsemi}.
Second, $\mathbf{S}^+_{/\mathcal{C}}$ is a $(\mathbf{S},\mathrm{QCat})$-enriched model category, the object
$\mathrm{Un}_{\mathcal{C}}(G)$ is fibrant by assumption and all its objects are cofibrant. Thus, to show that the map
$(1\times\pi_{\bar{q}_{(\Delta^{\bullet})^{\flat}}})^{\ast}$ is a categorical equivalence it suffices to show that the map
$1\times\pi_{\bar{q}_{(\Delta^{n})^{\flat}}}\colon\mathrm{Un}_{\mathcal{C}}(F)\times\pi_{(\Delta^n)^{\flat}}\rightarrow\mathrm{Un}_{\mathcal{C}}(F)\times\pi_{\mathrm{Un}_{\ast}((\Delta^n)^{\flat})}$ is a cartesian equivalence for every $n\geq 0$. Therefore in turn 
it suffices via \cite[Proposition 3.1.4.2]{luriehtt} to show that
$\bar{q}\colon(\Delta^n)^{\flat}\rightarrow\mathrm{Un}_{\ast}((\Delta^n)^{\flat})$ is a cartesian equivalence in $\mathbf{S}^+$ for 
every $n\geq 0$. This however is given by Corollary~\ref{corunstraight1} (or even directly by Lemma~\ref{lemmaunstraight1} as
$\Delta^n$ is a quasi-category with no non-trivial equivalences, so the marked simplicial set
$(\Delta^n)^{\flat}=(\Delta^n)^{\natural}$ is fibrant in $\mathbf{S}^+$).
%Thus, $(\Delta^n)^{\flat}=(\Delta^n)^{\natural}$ is a fibrant marked simplicial set. It follows that $\mathrm{Un}_{\ast}((\Delta^n)^{\flat})$ is fibrant as well, and that
%$\bar{q}_{(\Delta^n)^{\flat}}$ is a cartesian equivalence by Lemma~\ref{lemmaunstraight1}.
%
%\begin{align}
%\notag \mathbf{Fun}(\mathfrak{C}(\mathcal{C})^{op},\mathbf{S}^+)(F,G) & \cong U(\mathbf{Fun}(\mathfrak{C}(\mathcal{C})^{op},\mathbf{S}^+)(F\otimes\Delta^{\bullet},G)^{\simeq}) \\
%\notag & \cong U(\mathbf{Fun}(\mathfrak{C}(\mathcal{C})^{op},\mathbf{S}^+)(F\times c_{\Delta^{\bullet}},G)^{\simeq}) \\
% & \rightarrow U(\mathbf{S}^+_{/\mathcal{C}}(\mathrm{Un}(F\times c_{\Delta^{\bullet}}),\mathrm{Un}(G))^{\simeq} \\
%\end{align}
\end{proof}

\begin{corollary}\label{corcotensors}
The Unstraightening functor exhibits a binatural transformation
\begin{align}\label{equcorcotensors}
\gamma\colon\mathrm{Un}_{\mathcal{C}}(F^I)\rightarrow\mathrm{Un}_{\mathcal{C}}(F)^I
\end{align}
between the respective simplicial cotensors for $F\colon\mathfrak{C}(\mathcal{C})^{op}\rightarrow\mathbf{S}^+$ and $I\in\mathbf{S}$, 
which is a cartesian equivalence whenever $F$ is projectively fibrant. 
\end{corollary}

\begin{proof}
The composite binatural transformation
\[I\otimes\mathrm{Un}_{\mathcal{C}}(F^I)\cong\pi_{I^{\flat}}\times\mathrm{Un}_{\mathcal{C}}(F^I)\xrightarrow{\pi_{\bar{q}_{I^{\flat}}}\times 1}\pi_{\mathrm{Un}_{\ast}(I^{\flat})}\times\mathrm{Un}_{\mathcal{C}}(F^I)\cong\mathrm{Un}_{\mathcal{C}}(c_{I^{\flat}})\times
\mathrm{Un}_{\mathcal{C}}(F^I)\xrightarrow{\mathrm{Un}_{\mathcal{C}}(\mathrm{ev}_I)}\mathrm{Un}_{\mathcal{C}}(F)\]
yields a binatural transformation $\gamma$ as in (\ref{equcorcotensors}) via the respective tensor/cotensor adjunction. We are to 
verify that $\gamma$ is a cartesian equivalence whenever $F$ is projectively fibrant. Therefore, as $\gamma$ is natural in both 
arguments, as both model categories are $(\mathbf{S},\mathrm{QCat})$-enriched, and as $\mathrm{Un}_{\mathcal{C}}$ is right Quillen, we 
may assume without loss of generality that $I$ is a quasi-category. 
Furthermore, since $\mathrm{Un}_{\mathcal{C}}$ is essentially surjective, it suffices to show for every projectively bifibrant 
object $G\in\mathrm{Fun}(\mathfrak{C}(\mathcal{C})^{op},\mathbf{S}^+)$ that the functor 
\[\gamma_{\ast}\colon\mathbf{S}^+_{/\mathcal{C}}(\mathrm{St}_{\mathcal{C}}(G),\mathrm{Un}_{\mathcal{C}}(F^I))\rightarrow\mathbf{S}^+_{/\mathcal{C}}(\mathrm{St}_{\mathcal{C}}(G),\mathrm{Un}_{\mathcal{C}}(F)^I)\]
is an equivalence of quasi-categories. To do so we consider the following commutative diagram which commutes by construction of
$\gamma$.
\[\xymatrix{
\mathbf{Nat}(G,F^I)\ar[r]^(.4){\mathbf{Un}_{\mathcal{C}}}\ar[d]_[@!90]{\cong} & \mathbf{S}^+_{/\mathcal{C}}(\mathrm{Un}_{\mathcal{C}}(G),\mathrm{Un}_{\mathcal{C}}(F^I))\ar[r]^{\gamma_{\ast}} & \mathbf{S}^+_{/\mathcal{C}}(\mathrm{Un}_{\mathcal{C}}(G),\mathrm{Un}_{\mathcal{C}}(F)^I)\ar[d]^[@!-90]{\cong}\\
\mathbf{Nat}(I\otimes G,F)\ar[r]_(.4){\mathbf{Un}_{\mathcal{C}}} & \mathbf{S}^+_{/\mathcal{C}}(\mathrm{Un}_{\mathcal{C}}(I\otimes G),\mathrm{Un}_{\mathcal{C}}(F))\ar[r]_{(\pi_{\bar{q}_{I^{\flat}}}\times 1)^{\ast}} & \mathbf{S}^+_{/\mathcal{C}}(I\otimes\mathrm{Un}_{\mathcal{C}}(G),\mathrm{Un}_{\mathcal{C}}(F))
}\]
The two vertical actions $\mathbf{Un}_{\mathcal{C}}$ are equivalences of quasi-categories by Proposition~\ref{propunstraight1}. The map 
$\pi_{\bar{q}_{I^{\flat}}}\times 1\colon\pi_{I^{\flat}}\times\mathbf{Un}_{\mathcal{C}}(G)\rightarrow\pi_{\mathrm{Un}_{\ast}(I^{\flat})}\times\mathbf{Un}_{\mathcal{C}}(G)$ is a cartesian equivalence (between cofibrant objects) by Corollary~\ref{corunstraight1} and 
the fact that $\mathrm{Un}_{\mathcal{C}}(G)\in\mathbf{S}^+_{/\mathcal{C}}$ is fibrant. Thus, $\gamma_{\ast}$ is an equivalence by
2-out-of-3.
\end{proof}

\begin{remark}\label{remcorstraight1}
The argument in the proof of Corollary~\ref{corcotensors} that reduces arbitrary simplicial cotensors to quasi-categorical ones may 
appear dodgy on first sight, given that Corollary~\ref{corunstraight1} is central to the proof but has been shown for quasi-categories 
only. Against the background of Proposition~\ref{propunstraight1} however one can generalize Corollary~\ref{corcotensors} to apply to 
all simplicial sets directly as well.
\end{remark}

\section{Unstraightening as a weak cosmological biequivalence}\label{seccosmo}

The notion of an $\infty$-cosmos as defined in \cite{riehlverityelements} introduces a powerful framework for the 
abstract study of formal $\infty$-category theory with a wide range of applications. 
We recall that an $\infty$-cosmos is a $(\mathbf{S},\mathrm{QCat})$-enriched fibration category $\mathbf{K}$ with countable sequential 
limits of fibrations and small products in which all objects are cofibrant (\cite[Definition 1.2.1]{riehlverityelements}). 
These $\infty$-cosmoses can be understood as presentations of structurally well enough behaved $(\infty,2)$-categories similarly how 
simplicial model categories are understood as presentations of structurally well enough behaved $\infty$-categories. Case in point, 
every $(\mathbf{S},\mathrm{QCat})$-enriched model category $\mathbf{M}$ in which all fibrant objects are cofibrant gives rise to an 
underlying $\infty$-cosmos $\mathbf{K}:=\mathbf{M}^f$ of fibrant objects. In particular, for every quasi-category $\mathcal{C}$, the 
simplicial categories $\mathbf{Cart}(\mathcal{C})$ and
$\mathbf{Fun}(\mathfrak{C}(\mathcal{C})^{op},(\mathbf{S}^+,\mathrm{Cart})),\mathrm{Inj})^f$ come equipped with the structure of an
$\infty$-cosmos. 
Riehl and Verity furthermore define the notion of a cosmological functor between $\infty$-cosmoses; that is, a simplicially enriched 
functor which preserves fibrations as well as all the ``cosmological'' limits of \cite[Definition 1.2.1.(i)]{riehlverityelements} up 
to natural isomorphism (\cite[Definition 1.3.1.]{riehlverityelements}).

Given that $\infty$-cosmoses serve as a formal environment for the practice of $\infty$-category theory (at times referred to as 
``synthetic'' $\infty$-category theory), cosmological functors serve as handy compilers between such for a plethora of
$\infty$-categorical structures (\cite[Proposition 10.1.4]{riehlverityelements}). And given that Unstraightening constitutes an equivalence 
between indexed and fibered $\infty$-category theory as stated in the Introduction, one might want to use the Unstraightening functor 
as such a compiler. However, as this Unstraightening functor does not preserve simplicial cotensors on 
the nose, it cannot be cosmological. Therefore, one may introduce the following definition.

\begin{definition}\label{defquasicosmfunctor}
Let $G\colon\mathbf{K}\rightarrow\mathbf{V}$ be a simplicially enriched functor between $\infty$-cosmoses. Say that $G$ is
\emph{pseudo-cosmological} if
\begin{enumerate}
\item $G$ preserves fibrations, 
\item $G$ preserves products, pullbacks of fibrations and countable sequential limits of fibrations (that is, all ordinary cosmological limits of \cite[Definition 1.2.1.(i)]{riehlverityelements}), and 
\item $G$ preserves simplicial cotensors up to natural equivalence. That means, there is a simplicially enriched binatural 
transformation $\gamma_{(C,J)}\colon G(C^J)\rightarrow G(C)^{J}$ that is pointwise an equivalence in $\mathbf{V}$.
\end{enumerate}
Say that $G$ is a \emph{pseudo-cosmological embedding} if $G$ furthermore induces local equivalences of hom-quasi-categories.
Say that $G$ is a \emph{pseudo-cosmological biequivalence} if $G$ furthermore is essentially surjective on homotopy 
categories. 
\end{definition}

Clearly every cosmological functor/embedding/biequivalence between $\infty$-cosmoses is a pseudo-cosmological
functor/embedding/biequivalence. The notions are furthermore closed under composition. By Theorem~\ref{thmunstraight+} they also 
include the following example.

\begin{corollary}
The restricted Unstraightening construction
\[\mathbf{Fun}(\mathfrak{C}(\mathcal{C})^{op},(\mathbf{S}^+,\mathrm{Cart})),\mathrm{Inj})^f\hookrightarrow\mathbf{Fun}(\mathfrak{C}(\mathcal{C})^{op},(\mathbf{S}^+,\mathrm{Cart})),\mathrm{Proj})^f\xrightarrow{\mathbf{Un}_{\mathcal{C}}}\mathbf{Cart}(\mathcal{C})\]
defined on the $\infty$-cosmos of injectively fibrant marked simplicial presheaves is a pseudo-cosmological biequivalence of
$\infty$-cosmoses.
\end{corollary}\qed

\begin{conclusion*}
By \cite[Proposition 10.3.6]{riehlverityelements} cosmological biequivalences do not only preserve but also reflect and create a 
plethora of internal $\infty$-categorical properties and structures. And it is straight-forward (but a little lengthy) to recover
\cite[Propositions 10.1.4 and 10.3.6]{riehlverityelements} regarding the transfer of equivalence classes of these
$\infty$-cosmological structures between $\infty$-cosmoses along cosmological biequivalences for this weaker pseudo-cosmological 
notion as well. That means, one can show that such functors preserve (and such biequivalences furthermore create/reflect) equivalence 
classes of all the $\infty$-cosmological structures/properties which are listed in \cite[Proposition 10.3.6]{riehlverityelements} in 
1-1 correspondence. This for instance implies that Unstraightening yields an equivalence between 
fibered and indexed adjunctions, reflective localizations, limits and colimits, modules etc.\ over any given base $\mathcal{C}$.
\end{conclusion*}

\bibliographystyle{amsalpha}
\bibliography{BSBib}
%\nocite{*}
\Address
\end{document}